\documentclass{amsart}
\usepackage{amsmath}
\usepackage{amsbsy}
\usepackage{amssymb,latexsym} 
\usepackage{graphicx}
\usepackage{epstopdf}
\usepackage{color}
\newtheorem{theorem}{Theorem}
\newtheorem{definition}{Definition}

\newtheorem{proposition}{Proposition}
\newtheorem{lemma}[proposition]{Lemma}
\newtheorem{corollary}[proposition]{Corollary}
\newtheorem{remark}{Remark}

\begin{document}
\null
\vspace{-2cm}

 \vspace{2cm}
 \title{Geometric classification of real ternary  octahedral quartics }

\author {N. C. Combe }

 \address{Aix-Marseille University, IML, Campus de Luminy case 907 13288 Marseille Cedex-9 France}
 
 \email{noemie.combe@univ-amu.fr}
 
 \thanks{This work has been carried out in the framework of the Labex Archim\`{e}de (ANR-11-LABX-0033) and of the A*MIDEX project (ANR-11-IDEX-0001-02), funded by the ``Investissements d'Avenir'' French Government programme managed by the French National Research Agency (ANR)} 



\begin{abstract}
Ternary real-valued quartics in $\mathbb{R}^{3}$ being invariant under octahedral symmetry are considered. The geometric classification of these surfaces is given. A new type of surfaces emerge from this classification.
\end{abstract}

\maketitle

\section{Introduction}
 

 \vspace{3pt}
In this paper we are interested in  classifying real quartic surfaces, in the three-dimensional affine space, admitting an octahedral symmetry. More precisely, given a real ternary polynomial of degree four invariant under the octahedral group, we discuss the following problem: what geometric and topologic characterization can one give on the 0-locus of these quartic forms?

In particular, the consideration of this problem leads to an investigation of the number of connected components of quartic octahedral surfaces. 
Moreover, from this entire classification emerge surfaces that were unlisted until now.
\vspace{3pt}

To classify such surfaces we use the same type of group theory approach that was  initiated by  E. Goursat ~\cite{Gou1887} and extended to the case of icosahedral symmetry by W. Barth~\cite{Ba}. In our proof we will use the quadric's classification given by T. J. I. Bromwich~\cite{Bro1905} and R. S. Buringthon~\cite{Bur1932} and matrix theory.
\vspace{3pt}

\vspace{3pt}
Let us introduce some necessary definitions.

\begin{definition}
An octahedral invariant real quartic surface  is the 0-locus of a real polynomial $f(x,y,z)$ of   degree 4 in $\mathbb{R}^3$, invariant by the octahedral group $O_{h}$. In other words
\[\forall g \in O_{h}\subset Gl(3,\mathbb{R}),\  f(g^{-1}{\bf v})=f({\bf v}),\ {\bf v}=(x,y,z).\]
\end{definition}
\vspace{3pt}

The group $O_{h}$ is the common symmetry group of  two Platonic solids : octahedron and cube, both being dual to each other. This is a group of order 48 which admits 9 symmetry planes: 3 parallel to the faces of the cube and 6 planes containing two opposite edges. Moreover, there are 3  axes of order 4 through the opposite vertices of the octahedron,  4 axes of order 3 through the opposite vertices of the cube and 6  axes of order 2 through the center of the edges of the solids.  
\vspace{3pt}

Let $Ox,Oy,Oz$ be three orthogonal coordinate axes in $\mathbb{R}^{3}$. Choose 
the regular octahedron $|x|+|y|+|z|=1$ with vertices $(\pm1,0,0), (0,\pm1,0) ,(0,0,\pm1)$ and the cube with vertices $\{(a,b,c)| a,b,c=\pm 1\}$. The group $O_{h}$   is generated by two rotations  $g_1=\left(\begin{smallmatrix}1&0&0\\
0&0&-1\\
0&1&0\\\end{smallmatrix}\right)$, $g_2=\left(\begin{smallmatrix} 0&0&-1\\
0&1&0\\
1&0&0\\\end{smallmatrix}\right)$ around the x- and y-axis by $\frac{\pi}{2}$ and the reflection $g_{3}=\left(\begin{smallmatrix}1&0&0\\
0&1&0\\
0&0&-1\\ \end{smallmatrix}\right)$ against the ($z=0$)-plane.

The ternary quartic form with real coefficients has, under octahedral group $O_{h}$, independent invariants of degree two, four and six. 
 These linearly independent homogeneous polynomials of lowest degree are:\begin{equation}\label{E:invocta}
\begin{aligned}
u&=x^2+y^2+z^2 ,\\
v&=x^2y^2+y^2z^2+z^{2}x^{2},\\
w&=x^{2}y^{2}z^{2}. \\
\end{aligned}
\end{equation}
These forms generate the $O_{h}$ invariant polynomials, and by linear combination of these invariants one can write a polynomial $F(u,v,w)$ in variables $u,v,w$. As we are interested in the ternary real polynomial of degree 4,
we denote by $Q_{(A,B,C)}(D)$ an $O_{h}$-invariant quartic surface which will be represented in affine coordinates as the sum of quartic and quadratic forms,  by the equation: 
\begin{equation}\label{E:invquartic}
f(x,y,z)= A(x^2y^2+y^2z^2+x^2z^2)+B(x^2+y^2+z^2)^{2}+C(x^2+y^2+z^2)+D=0.
\end{equation}
Since we are looking at real quartic surfaces, the coefficients $A$ and $B$ cannot vanish altogether. As the 0-locus of the polynomials $f$ and $af, \ a \in \mathbb{R}$ are the same, the class of equivalent polynomials under change of coordinates  can be reduced to the study of the one-parameter family of equations determined by the following coefficients:
\[
A=0,\ \ \ B=1,\ \ C=\begin{cases} -1&\\\phantom{-}0&\\\phantom{-}1& \end{cases};\qquad
A=1, \ B=\begin{cases}  -b&\\ \phantom{-}0,&b>0\\ \phantom{-}b&\\
\end{cases}, \ C=\begin{cases} -1& \\\phantom{-}0\\\phantom{-}1& \end{cases}.\\
\]

\vspace{5pt}
\noindent{\bf Remark.}
We will consider a lattice in $\mathbb{R}^{3}$ generated by three linearly independent vectors: \[\bf{a_{1}}=(1,1,1), \bf{a_{2}}=(1,1,0), \bf{a_{3}}=(1,0,0).\] 
The idea of fundamental region, described by Gauss in 1831, sprung from the theory of Lattices. 
Hence one can note that the lattice in $\mathbb{R}^{3}$ gives rise to a tessellation of the space, generated by a fundamental polyhedron of the crystallographic group $O_{h}$. In our case the fundamental polyhedron is a polytope with 5 vertices, 6 edges and 6 faces. 
\vspace{3pt}

The symmetric properties of the quartic imply that it is sufficient to consider the surface in a fundamental polyhedron. However since we are interested in the geometry of the quartic, we can restrain the fundamental polyhedron to a half-fundamental polyhedron domain that is to a tetrahedron.
Thus we will consider the quartic real octahedral surfaces in the region: $ \{x\geq 0, 0\leq y\leq x, 0\leq z \leq  x \}$ .
\vspace{3pt}

\noindent  {\bf Acknowledgments:} I would like to thank Professor Daniel Coray for initiating me in this subject and for interesting discussion during my stay in Geneva.

\section{Classification of quartic surfaces admitting octahedral symmetry}
This section is devoted to giving classification theorems of the octahedral quartics.
We shall introduce some terminology before we start. 
\begin{definition}
Let $M_{1}, M_{2}$ be two manifolds, int $B_{i}$ the interior of a ball in $M_{i}$ and $h: \partial B_{1} \rightarrow \partial B_{2}$ a homeomorphism. A connected sum of two manifolds $M_{1}\#M_{2}$ is obtained by removing the interior of the balls  in each of the manifolds and by gluing the boundaries of the discs by the above homeomorphism$h$. $M_{1}\#M_{2}=(M_{1}-int(B_{1}) )\cup_{h}(M_{2}-int(B_{2}) )$ 
\end{definition}
This definition will be used to describe some particular quartics.

\vspace{3pt}
In order to have a complete classification of these quartics, we will consider systematically each of the different families of parameters. The aim is to have a list of all the different cases and results corresponding to them. In the first two theorems we will present  the elementary cases.

\vspace{3pt}

\subsection{Real quartic surface with A=0}
In the case where $A=0$ the quartic surface is the 0-locus of the polynomial
\begin{equation}\begin{aligned}f(x,y,z)&= (x^{2}+y^{2}+z^{2})^{2}+C(x^{2}+y^{2}+z^{2}) +D,\\
&=\left(x^{2}+y^{2}+z^{2}+\frac{C}{2}(1-k)\right)\left(x^{2}+y^{2}+z^{2}+\frac{C}{2}(1+k)\right),\ k=\sqrt{1-\frac{4D}{C^{2}} }\end{aligned}
.\end{equation}
Under change of coordinates we can limit ourself to the three one-parameter families labelled by $C=-1,0,1$ 


\begin{theorem}\label{T:01CD} Real quartic surfaces with octahedral symmetry  $Q_{(0,1,C)}(D)$ are,
\begin{enumerate}
\item  $A=0,B=1,C=-1.$ 
\begin{itemize}
\item $D<0$, $Q_{(0,1,-1)}(D) = S^{2}\left(0,\sqrt{\frac{1}{2}(1+k)}\right)$,
\item $D=0$, $Q_{(0,1,-1)}(0)=S^{2}(0,1)\sqcup \{0\}$, where $\{0\}$ is a singular point of the surface,

\item If $0<D<\frac{1}{4}$,   $Q_{(0,1,-1)}(D)=S^{2}\left(0,\sqrt{\frac{1}{2}(1-k)}\right)\sqcup S^{2}\left(0,\sqrt{\frac{1}{2}(1+k)}\right)$. Both concentric spheres are such that their normal vector field points outwards for the external sphere and points inwards for the internal one,
\item  if $D=\frac{1}{4} $,  $Q_{(0,1,-1)}(\frac{1}{4}) =[S^{2}\left(0,\frac{1}{\sqrt{2}}\right)]^{2}$ is a singular   sphere with multiplicity two.
\end{itemize}

\vspace{3pt}
\item $A=0,B=1,C=0.$
\begin{itemize}
\item $D<0$,  $Q_{(0,1,0)}(D) =S^{2}(0,\sqrt{-D})$.
\item $D=0$,  $Q_{(0,1,0)}(0) =\{0\}$.
\end{itemize}

\vspace{3pt}
\item  $A=0,B=1,C=1.$
\begin{itemize}
\item  $D<0$,  $Q_{(0,1,1)}(D)=S^{2}\left(0,\sqrt{\frac{1}{2}(k-1)}\right) $, 
\item $D=0$, $Q_{(0,1,1)}(0)=\{0\}$, singular point.
\end{itemize}
\end{enumerate}
\end{theorem}

\subsection{ Real quartic surfaces with $A=1,B=0$}
In the case where $A=1,B=0$ the quartic surface verifies the equation
\begin{equation}f(x,y,z)= x^{2}y^{2}+y^{2}z^{2}+z^{2}x^{2}+C(x^{2}+y^{2}+z^{2}) +D=0,\end{equation}
Under change of coordinates we can limit ourself to the three one-parameter families labelled by $C=-1,0,1$ 


\begin{theorem}\label{T:10CD}
Real quartic surfaces with octahedral symmetry  $Q_{(1,0,C)}(D)$ are,
\begin{enumerate}
\item  $A=1,B=0,C=-1$.
\begin{itemize}
\item $D<0$, stellated octahedron being the connected sum of six half-cylinders $\mathcal{W}_{i}$ with axes on the coordinate axes and the octahedron. $Q_{(1,0,-1)}(D)= \mathcal{E}_{-}\#_{i=1}^{8}\mathcal{W}_{i}$, 
\item $D=0$,  $Q_{(1,0,-1)}(0)= \mathcal{E}_{-}\sqcup \{0\}$ singular in $\{0\}$,
\item $0<D<\frac{3}{4}$, disjoint union of a cuboid and a stellated octahedron $Q_{(1,0,-1)}(D)=  \mathcal{E}_{-}\sqcup \mathcal{S}^{2}$,  $\mathcal{S}^{2}$ is a topological sphere nested in  $\mathcal{E}_{-}$,
\item $D=\frac{3}{4}$, singular surface with eight conical singularities at the vertices of a cube, dual to the octahedron.  $Q_{(1,0,-1)}(\frac{3}{4})$ 
\item $\frac{3}{4}<D<1$, multi-connected stellated surface being a connected sum of six half-cylinders and an internal cube. Remark that the removed discs are situated on the vertices of the cube and on each point joining a pair of half-cylinders. These points are positioned in the center of a regular octahedron.
\item  $D=1$,  $Q_{(1,0,-1)}(1)$, multi-connected surface with 12 isolated conical singular points, at the center of the edges of the octahedron,
\item  $D>1$, $Q_{(1,0,-1)}(D)$  This surface has 6 disjoint half-cylinders .
\end{itemize}

\vspace{3pt}
\item  $A=1,B=0,C=0.$
\begin{itemize}
\item $D<0$, $Q_{(1,0,0)}(D)$ Regular concave hyperbolic stellated surface. The surface has 6 branches converging to the coordinate axis,
\item $D=0$,  $Q_{(1,0,0)}(0)$ is the singular surface reduced to the 3 coordinate axis.
\end{itemize}

\vspace{3pt}
\item $A=1,B=0,C=1.$
\begin{itemize}
\item $D<0$, $Q_{(1,0,1)}(D)$ is a deformation of the octahedron in a regular compact surface bounded by a sphere of radius $\sqrt{-D}$ and tangent to this sphere at the coordinate axis, 
\item $D=0$, $Q_{(1,0,1)}(0)$  is the singular point $\{0\}$.
\item $0<D<\frac{4}{3}$ the surface is the disjoint union of two concentric components centered at the origin. One of which is compact and the other is a concave hyperbolic stellated surface.
\end{itemize}
\end{enumerate}
\end{theorem}

\subsection{ Real quartic surfaces with $A=1,B\ne0, C=0$}
In the case where $A=1,B<0$ the quartic surface verifies the equation
\begin{equation}f(x,y,z)= x^{2}y^{2}+y^{2}z^{2}+z^{2}x^{2}+B(x^{2}+y^{2}+z^{2})^{2} +D=0,\end{equation}
\begin{theorem}\label{T:1B0D}
Real quartic surfaces with octahedral symmetry  $Q_{(1,B,0)}(D)$ are,
\begin{enumerate}
\item If $B<-\frac{1}{3}$, 
\begin{itemize}
\item $D> 0$ The quartic is a cuboid (topologically sphere).
\item $D=0$,  the topological sphere degenerates to the point $\{0\}$.
\end{itemize}

\item If $B=-\frac{1}{3}$,
\begin{itemize}
\item if $D>0$,  the quartic  is a stellated cube. 
\item if $D=0$,  the  stellated cube degenerates to the point $\{0\}$.
\end{itemize}

\item If $-\frac{1}{3}<B<0$,
 \begin{itemize}
 \item $D< 0$, the quartic is hence composed of eight hyperbolic sheets at the vertices of a cube. 
\item $ D> 0$, the quartic is disjoint sum of six smooth cones with axis being the coordinates axis.
\item $D=0$,  the quartic degenerates to the point $\{0\}$.
\end{itemize}

\item If $B>0$ and $D<0$, the quartic is a compact topological sphere. 
\end{enumerate}
\end{theorem}

\subsection{ Real quartic surfaces with $A=1,B\ne0, C\ne 0$}
In the case where $A=1,B\ne 0$ the quartic surface verifies the equation
\[f(x,y,z)= x^{2}y^{2}+y^{2}z^{2}+z^{2}x^{2}+B(x^{2}+y^{2}+z^{2})^{2}+C(x^{2}+y^{2}+z^{2}) +D=0,\]
Under the change of coordinates, the real algebraic surface $Q_{(1,B,C)},\ C\ne0$ can be seen as the 0-locus of the polynomial 
\begin{equation}
\begin{aligned} f_{\epsilon_{1},\epsilon_{2}}(x,y,z)
&=\beta(x^2y^2+y^2z^2+x^2z^2)+\epsilon_{1}(x^2+y^2+z^2)^{2}+ \epsilon_{2}(x^2+y^2+z^2)+ \frac{DB}{C^2},\\
&=\beta(x^2y^2+y^2z^2+x^2z^2)+\epsilon_{1}\left(x^2+y^2+z^2+\frac{\epsilon_{2}\epsilon_{2}}{2}\right)^{2} -\frac{\epsilon_{1}}{4}\left(1-k\right)\\
\end{aligned}
\end{equation}
with \[
\beta=\frac{1}{b},\ b=|B|,\,\epsilon_{1}=\text{sgn}(B),\  \epsilon_{2}=\text{sgn}(C), \ \ k=\frac{4D}{C^2\beta}\leq 1.
 \]
 For convenience, the surface $Q^{\epsilon_{1},\epsilon_{2}}(k)$ will be noted as  $Q_{(1,B,C)}(D) )$.
 

\begin{theorem}\label{T:1bCD}
Real quartic surfaces with octahedral symmetry  $Q_{(1,b,C)}(D)$ are,
\begin{enumerate}
\item  $A=1,B=b,C=b,\ b>0$  ($\epsilon_{1}=1,\epsilon_{2}=1, k= 4\frac{ D}\beta,\ \beta=b^{-1}$)
\begin{equation}\label{E:++}
 f_{+,+}=\beta(x^2y^2+y^2z^2+x^2z^2)+\left(x^2+y^2+z^2+\frac{1}{2}\right)^{2}-\frac{1-k}{4}=0.
 \end{equation}
 v\begin{itemize}
\item $k<0$, $Q^{++}(k)$ smooth compact surface, topologically sphere. 
\item $k=0$,  $Q^{++}(0)=\{0\}$ is the singular point at the origin.
\end{itemize}

\vspace{5pt}
\item $A=1,B=b,C=-b,\ b>0$  ($\epsilon_{1}=1,\epsilon_{2}=-1, k= 4\frac{ D}\beta,\ \beta=b^{-1}$),
\begin{equation}f_{+,-}\label{E:+-}=\beta(x^2y^2+y^2z^2+x^2z^2)+\left(x^2+y^2+z^2-\frac{1}{2}\right)^{2}-\frac{1-k}{4}=0.
\end{equation}
\begin{itemize}
\item $k<0$, $Q^{+-}(k)=\overset{\smallfrown}{\mathcal{O}}$ is a smooth octahedron $\overset{\smallfrown}{\mathcal{O}}$, the surface is topologically a sphere.
\item $k=0$, $Q^{+-}(0)=\overset{\smallfrown}{\mathcal{O}}\sqcup\{0\}$.
\item $0<k<\frac{3}{3+\beta}$, $Q^{+-}(k) =\overset{\smallfrown}{\mathcal{O}}\sqcup \overset{\smallfrown}{\mathcal{S}q}$, disjoint union of a smooth  octahedron and a spherical cube $ \overset{\smallfrown}{\mathcal{S}q}$
\item $k=\frac{3}{3+\beta}$, $Q^{+-}(\frac{3}{3+\beta})$, singular surface with eight conical singularities at the vertices of the internal cube and the center of the face of the octahedron. 
\item $\frac{3}{3+\beta}<k<\frac{4}{4+\beta}$, $Q^{+-}(k)=\overset{\smallfrown}{\mathcal{O}}\#_{i=1}^{8}\mathcal{C}^{(\frac{1}{8})}_{i} $, is a double surface invariant under $O_{h} $, with eight holes  centered on the diagonal  of the cube obtained as the connected sum of each octant of the cube $\mathcal{C}^{(\frac{1}{8})}_{i} $ with $\overset{\smallfrown}{\mathcal{O}}$.
\item $k=\frac{4}{4+\beta}$, $Q^{+-}(\frac{4}{4+\beta})$, singular compact connected surface with twelve conical singularities at the center of edges of an octahedron. 
\item $\frac{4}{4+\beta}<k<1$, $Q^{+-}(k)$, is the disjoint union of connected compact components, each homeomorphic to a sphere. Each component is at the vertex of an octahedron. Moreover the component are vanishing for $k=1$.
\end{itemize}
\end{enumerate}
\end{theorem}


\begin{theorem}\label{T:1-bCD}
Real quartic surfaces with octahedral symmetry  $Q_{(1,-b,C)}(D), b>0$ are,
\begin{enumerate}
\item  $A=1,B=-b,C=-b,\ b>0$  ($\epsilon_{1}=-1,\epsilon_{2}=-1, k=4\frac{D}{\beta},\ \beta=b^{-1}$)
\begin{equation}\label{E:--}\begin{aligned}f_{-,-}&=\beta(x^2y^2+y^2z^2+x^2z^2)-(x^2+y^2+z^2)^{2}-(x^2+y^2+z^2)-\frac{k}{4} \\
&=\beta(x^2y^2+y^2z^2+x^2z^2)-\left(x^2+y^2+z^2+\frac{1}{2}\right)^{2}+\frac{1-k}{4}=0.\end{aligned}.\end{equation}

\begin{enumerate}
\item $0<\beta\leq3$,
\begin{itemize}
\item $k<\frac{3}{3-\beta}$, $Q_{--}(k)=\mathcal{C}$, cuboide vanishing in $\{0\}$. 
\item $k=0$,  $Q^{--}(0)$ has a singular point at $\{0\}$.
\end{itemize}

\item $\beta=3$
\begin{itemize}
\item $k<0$ the surface is an eight branched star 
\end{itemize}

\item $ 3<\beta<4$

\begin{itemize}
\item $k<\frac{3}{3-\beta}$,   $Q^{--}(k)=(\#_{i=1}^{8}\mathcal{H}_{i})\#\mathcal{C}$. The surface is the connected sum of a cube with hyperbolic sheets having the diagonal of the cube as axis.

\item $\frac{3}{3-\beta}=k$, $Q^{--}(\frac{3}{3-\beta})$, singular quartic, analogous to Cayley's cubic: cube with cones at each vertex.  This surface has eight conical singular points at the vertices of a cube.

\item $\frac{3}{3-\beta}<k<0$, $Q^{--}(k)=(\sqcup_{i=1}^{8}\mathcal{H}_{i})\sqcup\mathcal{C}$. 
\item $k=0$,  $Q^{--}(k)=(\sqcup_{i=1}^{8}\mathcal{H}_{i})\sqcup\{0\}$.
\item $\frac{3}{3-\beta}<0<k$, $Q^{--}(k)=\sqcup_{i=1}^{8}\mathcal{H}_{i}$ 

\end{itemize}
\item $ \beta=4$ 
\begin{itemize}
\item $k=\frac{3}{3-\beta}=-3$ the surface has eight singularities.
\item $k=\frac{4}{4-\beta}\to\infty$ the surface has twelve singularities.
 \item $k=0$ the surface has a singularity at $\{0\}$.
\end{itemize}
\item $ 4<\beta$
\begin{itemize}
\item $k<\frac{4}{4-\beta}$ the surface has six disjoint connected components.
\item $k=\frac{4}{4-\beta}$ the surface has twelve singularities; the surface is compact.
\item $\frac{4}{4-\beta}<k<\frac{3}{3-\beta}$  the surface is compact .
\item $k=\frac{3}{3-\beta}$ the surface has eight singular point and the surface is compact.
\item $\frac{3}{3-\beta}<k<0$ the surface has two concentric connected components
\item $k=0$, the surface is the disjoint union of a connected component and of a singularity at zero.
\item $0< k$ the surface has one connected component, centered at the origin.
\end{itemize}
\end{enumerate}

\vspace{5pt}
\item $A=1,B=-b,C=b,\ b>0$  ($\epsilon_{1}=-1,\epsilon_{2}=1,\ k= 4\frac{ D}{\beta},\ \beta=b^{-1}$),
\begin{equation}\label{E:-+} \begin{aligned}f_{-,+}&=\beta(x^2y^2+y^2z^2+x^2z^2)-(x^2+y^2+z^2)^{2}+(x^2+y^2+z^2)-\frac{k}{4}\\
&=\beta(x^2y^2+y^2z^2+x^2z^2)-\left(x^2+y^2+z^2-\frac{1}{2}\right)^{2}+\frac{1-k}{4}=0.\end{aligned}\end{equation}
\begin{enumerate}
\item $0<\beta<3$,

\begin{itemize}
\item $k<0$, 
   the quartic surface $Q^{-+}_{1}(k)=\mathcal{C}_{1}$, is a compact connected stellated cube, homeomorphic to a sphere. 
\item $k=0$, the quartic surface $Q^{-+}_{1}(0)=\mathcal{C}_{1}\sqcup \{0\}$ is the disjoint union  of the cube $\mathcal{C}_{1}$ and the singular point $\{0\}$.
\item $0<k<1$, 
the surface  $Q^{-+}_{1}(k)=\mathcal{C}_{1} \sqcup \mathcal{C}_{2}$ is the disjoint union of two concentric components $\mathcal{C}_{1},  \mathcal{C}_{2}$, one contained in the other $\mathcal{C}_{2}\subset\mathcal{C}_{1}$, each of  these components is topologically a sphere.
\item $k=1$, $Q^{-+}_{1}(1)$ is a singular surface which is the union of an octahedron embedded in a cuboid. These two surfaces touch in six conical singular points. Each singular point is situated at the vertices of the octahedron and joins each vertex to the center of each cuboid's face.
\item$1<k<\frac{4}{4-\beta}$,
 $Q^{-+}_{1}(\frac{4}{4-\beta})$, is a double cuboid with six holes, one in each face of cuboid. This surface is multi-connected. We remove six discs from each face of both concentric cuboids and glue the boundaries of the parallel discs together.   
\item $k=\frac{4}{4-\beta}$, singular surface with twelve conical singularities at the center of each edge of the cuboid. 
\item $\frac{4}{4-\beta}<k<\frac{3}{3-\beta}$, $Q^{-+}_{1}(k)$ is the disjoint union of eight compact connected components, homeomorphic to spheres. They vanish when $k$ tends to 
$\frac{3}{3-\beta}$.
\end{itemize}

\vspace{3pt}
\item $\beta=3$,

\begin{itemize}
\item $k<0$, $Q^{-+}_{2}(k)$, is a stellated cube, connected sum of eight half-cylinders $\mathcal{W}_{i}$ with axes being the diagonal of a cube and the cube $\mathcal{C}$, $\mathcal{C}\#_{i=1}^{8}\mathcal{W}_{i}=$.
\item $k=0$, 
$Q^{-+}_{2}(0)=(\mathcal{C}\#_{i=1}^{8}\mathcal{W}_{i})\sqcup \{0\}$,
\item $0<k<1$, $Q^{-+}_{2}(k)=(\mathcal{C}\#_{i=1}^{8}\mathcal{W}_{i})\sqcup \mathcal{O}$, the disjoint union of an octahedron $\mathcal{O}$ contained in a stellated cube.
\item $k=1$,$Q^{-+}_{2}(1)$, is a singular double surface union of an octahedron $\mathcal{O}$ embedded in $\mathcal{C}\#_{i=1}^{8}\mathcal{W}_{i}$ with six   conic singular points. Each singular point  at the vertices of an octahedron joints the center of each stellated cube's faces.
\item $1<k<4$, 
$Q^{-+}_{2}(k)=\mathcal{O}\#(\mathcal{C}\#_{i=1}^{8}\mathcal{W}_{i}) $, the connected sum of an octahedron $\mathcal{O}$ with the connected sum of eight half-cylinders $\#_{i=1}^{8}\mathcal{W}_{i}$.
\item $k=4$, 
$Q^{-+}_{2}(4)$, Kummer's surface with twelve conic singularities at the center of the edges of a cube. The surface consists of eight closed half-cylinders, the axis of which are the diagonals of the cube. Each half-cylinder touches the three nearest neighboring half-cylinders at singular points.
\end{itemize}

\vspace{3pt}
\item $3< \beta<4$,
\begin{itemize}
\item $k<0$, 
$Q^{-+}_{3}(k)=\mathcal{C}{\#}_{i-1}^{8}\mathcal{K}_{i}$, cube stellated by cones, connected sum of eight cones $\mathcal{K}_{i}$ with axes being the diagonal of a cube and the cube $\mathcal{C}$
\item $k=0$,  
the quartic $Q^{-+}_{3}(0)=(\mathcal{C}\#_{i-1}^{8}\mathcal{K}_{i})\sqcup \{0\}$,  disjoint union of a stellated cube  $\mathcal{C}\#_{i=1}^{8}\mathcal{K}_{i}$ and the singular point $\{0\}$.
\item $0<k<1$, the quartic $Q^{-+}_{3}(k)=(\mathcal{C}\#_{i=1}^{8}\mathcal{K}_{i})\sqcup\mathcal{O}$. 
\item $k=1$, $Q^{-+}_{3}(1)=(\mathcal{C}\#_{i=1}^{8}\mathcal{K}_{i})\cup\mathcal{O}$, singular  with six conic singularities at the vertices of the octahedron and center of faces of cube.
 \item $1<k<\frac{4}{4-\beta}$, 
$Q^{-+}_{3}(k)= (\mathcal{C}\#_{i=1}^{8}\mathcal{K}_{i})\# \mathcal{O}$, muticonnected surface with six holes.
\item $k=\frac{4}{4-\beta}$, 
$Q^{-+}_{3}(\frac{4}{4-\beta})$, Kummer's surface with twelve singular conic points in the center of edges of a cube.
\item $k>\frac{4}{4-\beta}$, 
$Q^{-+}_{3}(k)=\sqcup_{i=1}^{6} \mathcal{T}_{i}$, eight disjoint hyperbolic sheets with section of triangular type $\mathcal{T}_{i}$ and axis being the diagonals of a cube. 
\end{itemize}

\vspace{3pt}
\item $ \beta=4$,
\begin{itemize}
\item $k<0$, $Q^{-+}_{4}(k)=\sqcup_{i=1}^{6} \mathcal{S}q_{i}$, six disjoint hyperbolic sheets with section of square's type $ \mathcal{S}q_{i}$, and axes being the coordinate axes{\ss}\dag.
 \item $k=0$, $Q^{-+}_{4}(0) =\sqcup_{i=1}^{6} \mathcal{S}q_{i}\sqcup \{0\}$,
 \item $0<k<1$,  $Q^{-+}_{4}(k)= \sqcup_{i=1}^{6} \mathcal{S}q_{i}\sqcup \mathcal{O}$ is the disjoint union of six hyperbolic sheets $\mathcal{S}q_{i}$  and an octahedron centered at the origin $\{0\}$. 
 \item $k=1$, $Q^{-+}_{4}(1)$, singular connected surface with six conic singularities at the vertices of an octahedron and the hyperbolic sheets $\mathcal{S}q_{i}$ with axes being the coordinate axes.
\item $k>1$, $Q^{-+}_{4}(k)=\mathcal{O}\#_{i=1}^{6} \mathcal{S}q_{i}$, connected surface with six holes.
\end{itemize}
 
\vspace{3pt}
\item $ \beta\geq 4$,
\begin{itemize}
\item $k<0$,  $Q^{-+}_{5}(k)=\sqcup_{i=1}^{6} \mathcal{S}q_{i}$
 \item $k=0$, $Q^{-+}_{5}(0)=\sqcup_{i=1}^{6} \mathcal{S}q_{i}\sqcup\{0\}$.
 \item $0<k<1$, $Q^{-+}_{5}(k)=\sqcup_{i=1}^{6} \mathcal{S}q_{i}\sqcup\mathcal{O}$.
 \item $k=1$, $Q^{-+}_{5}(k)$ singular connected surface with six conic singularities at the vertices of an octahedron and the hyperbolic sheets $\mathcal{S}q_{i}$ with axes being the coordinate axes.
 \item $k>1$, $Q^{-+}_{5}=\mathcal{O}\#_{i=1}^{6}\mathcal{S}q_{i}$.
\end{itemize}
\end{enumerate}

\end{enumerate}
\end{theorem}


\section{Proof of the theorems }


\subsection{Method of proof}
\vspace{3pt}
 
Consider the equation of the octahedral quartic~\eqref{E:invquartic}:
\[f(x,y,z)= A(x^{2}y^{2}+y^{2}z^{2}+z^{2}x^{2})+B(x^{2}+y^{2}+z^{2})^{2}+C(x^{2}+y^{2}+z^{2}) +D=0.\]
Using a change of variables $x^2=X,\ Y= y^2,\ Z=z^2$, the equation~\eqref{E:invquartic} is replaced by the quadric equation : 
\begin{equation}\label{E:quarticequiv}
F(X,Y,Z)= A(XY+YZ+XZ)+B(X+Y+Z)^{2}+C(X+Y+Z)+D=0,
\end{equation}
defined on the new metric space: \[dX=dx^2, dY=dy^2 , dZ=dz^2,\  \ \hspace{5pt}  \\ ds^2= \frac{dX^2}{4X}+\frac{dY^2}{4Y} +\frac{dZ^2}{4Z}\  \hspace{5pt} .  \] 
 with metric tensors defined by:
\[
\left( \begin{array}{ccc} dX& dY&dZ \end{array} \right)
\left( \begin{array}{ccc}
 \frac{1}{4 X} & 0 &0  \\ 
 0 & \frac{1}{4 Y} & 0 \\
0 & 0 & \frac{1}{4 Z} 
\end{array} 
\right) \left( \begin{array}{c} dX \\ dY \\ dZ \end{array} 
\right)
\mbox{~and~} \ \hspace{3pt} g=\det\left( \begin{array}{ccc}
 \frac{1}{4 X} & 0 &0  \\ 
 0 & \frac{1}{4 Y} & 0 \\
0 & 0 & \frac{1}{4 Z} 
\end{array}  \right)=\frac{1}{64XYZ} >0. \]

This change of variable enables us to use Bromwich-Buringthon's classification~\cite{Bro1905,Bur1932,MeYa} of quadrics in three affine space~(see also appendix~\ref{A:BrBu}).
\vspace{3pt}

Note that the degree two surface ~\eqref{E:quarticequiv} has to be  defined in the so-called first octant: $ \{X\geq0 ,Y\geq 0 ,Z \geq 0 \}$ and that it's topological properties, in the first octant, will stay unchanged after using the inverse transformation map to the standard euclidean metric.

\vspace{3pt}
 The corresponding quartic surface will be obtained from the inverse transformation, limited to the first octant, and from the subgroup of $O_{h}$ , generated by the reflections about the coordinate planes.  However, to characterize the quadric~\eqref{E:quarticequiv} it  is sufficient to know it's properties in an even smaller domain than previously:  Let us consider the plane arrangement $X=0,Y=0,Z=0,X=Y,X=Z,Y=Z$ of $\mathbb{R}^{3}$. This arrangement cuts space into 48 unbounded congruent cones, whose intersection  with the unit sphere furnishes a tessellation (2,3,4) of the sphere~\cite{Rat}. Note that every fundamental domain of this tessellation is a spherical triangle with angles $(\frac{\pi}{2},\frac{\pi}{3},\frac{\pi}{4})$. The  quartic associated to the quadric~\eqref{E:quarticequiv} is obtained by the inverse transformation, limited to one of the six cones of the first octant, for example $\{X\geq 0,\ 0\leq Z \leq X,\ 0 \leq Z\leq Y\}$, and use of the $O_{h}$ group. 
 \begin{definition}
 We will call a fundamental domain the region delimited by the planes $Z=0,\ X=Y,\\ Z=X,\ $ in the first octant, which is one of the 48 regions of the plane arrangement.
 \end{definition}

To study the quadric~\eqref{E:quarticequiv} in terms of Bromwich-Buringthon's classification it is convenient to rewrite the polynomial  as a matrix product: 
 \begin{equation}\label{E:Mquarticequiv}
 \mathbf{ X}^{T}\Lambda\ \mathbf{ X}=0,\quad \mathbf{ X}=(1,X,Y,Z)^{T}, \end{equation}
 with
\[ \Lambda=\left(\renewcommand{\arraystretch}{1.5}\begin{tabular}{c|c}
$D$ &$ \frac{1}{2}\mathbf{C}^{T}$\\ \hline
$\frac{1}{2}\mathbf{C}$ & $\Lambda_{0}$  
\end{tabular}\right),
 \quad \mathbf{C}=(C,C,C)^{T},\ \Lambda_{0}= \begin{pmatrix}
B & W  & W\\ 
W & B & W\\
W & W & B  \\  
\end{pmatrix},\quad W=\frac{A+2B}{2}.
 \]
 
 We have
 \begin{equation}
\det \Lambda = (B-W)^{2}\left((B+2W)D-\frac{3}{4}C^{2} \right)   ,\quad \det \Lambda_{0}=(B-W)^{2}(B+2W).
 \end{equation}
 
 The eigenvalues of $\Lambda_{0}$ are:
  \begin{equation}
  \lambda_{1}=\lambda_{2}=B-W=-\frac{A}{2}, \ \lambda_{3}=B+2W= A+3B\}.
  \end{equation}
 From the normalized  eigenvectors $\{\mathbf{v}_{1},\mathbf{v}_{2},\mathbf{v}_{3}\}$, one can obtain the columns of  the  transition matrix $P$ : 
\[ P=\left( \begin{array}{ccc}
-\frac{1}{\sqrt2} & -\frac{1}{\sqrt2}  & \frac{1}{\sqrt3}\\ 
0 & \frac{1}{\sqrt2} & \frac{1}{\sqrt3}\\
\frac{1}{\sqrt2} & 0 & \frac{1}{\sqrt3}  \\  
\end{array}\right).\]
\begin{remark}
  The eigen-plane of the double eigenvalues  $\lambda_{1}= \lambda_{2}=-\frac{A}{2}$ is orthogonal to the axis  which corresponds  to simple eigenvalue $\lambda_{3}=A+3B$, so the quartic has this axis as revolution axis. 
 \end{remark}

To determine the  the center $P_{0 }$ of the quadric, let us rewrite equation~\eqref{E:Mquarticequiv}
\[\ \mathbf{ X}_{0}^{T}\Lambda_{0}\ \mathbf{ X}_{0}+\mathbf{C}^{T}\mathbf{ X}_{0} +D=0,\quad \mathbf{ X}_{0}^{T}=(X,Y,Z).\] 
Let $\mathbf{X}=P\mathbf{X'}$, then 
\[\begin{aligned}\mathbf{X'}\cdot\ \text{diag}(\lambda_{1},\lambda_{2},\lambda_{3}) \cdot \mathbf{X'}^{T}+ \mathbf{C}^{T} \cdot P \cdot \mathbf{X'} +D&=0\\
 (B-W)({X'}^{2} +{Y'}^{2})+(B+2W)\left ({Z'} +\frac{C\sqrt{3}}{2(B+2W)}\right)^{2} &= \frac{3C^{2}}{4(B+2W)^{2}}-D, \end{aligned}\]
Then the center of the quadric in the $\mathbf{X}'$ coordinates lies on the $\mathbf{v}_{3}$ axis at  the point $ X'=Y'=0,\ Z'= -\frac{C \sqrt{3}}{2(B+2W)} $, and in $\{\mathbf X\}$ coordinate  is
$\left(-\frac{C}{2(B+2W)},-\frac{C}{2(B+2W)},-\frac{C}{2(B+2W)}\right) $,
\begin{equation}
P_{0}=-\frac{C \sqrt{3}}{2(B+2W)}\mathbf{v}_{3}.
\end{equation}


\subsection{Proof of theorem~\ref{T:01CD}}
As we have seen before when $A=0$ we can choose $B=1$, and $C=-1,0$ or 1, then
\[ \Lambda=\left(\renewcommand{\arraystretch}{1.5}\begin{tabular}{c|c}
$D$ &$ \frac{1}{2}\mathbf{C}^{T}$\\ \hline
$\frac{1}{2}\mathbf{C}$ & $\Lambda_{0}$  
\end{tabular}\right),
 \quad \mathbf{C}=(C,C,C)^{T},\ \Lambda_{0}= \begin{pmatrix}
1 & 1 & 1\\ 
1 & 1 &1\\
1 & 1 & 1  \\  
\end{pmatrix}.
 \]
 then $\text{rk}(\Lambda)=2$ and $\det \Lambda_{0}=0$. By Bromwich-Burighton's classification,  this quadric belongs to the case of parallel planes. The position of the planes with respect to the first octant can be given by the number of intersections with the half-line $X=Y=Z>0$, that is the number of positive roots of the equation
\begin{equation} \label{E: Intdigsurf} 3X^2+3CX+D=0.\end{equation} 
A direct discussion shows the domain of  D where there is 0,1 or 2 intersections. This number of intersections depends on $C=-1,0,1$. By inverse transformation and symmetry we obtain theorem~\ref{T:01CD}.\qed


\subsection{Proof of theorem~\ref{T:10CD}}
In the case $A\ne 0$ and $B=0$ we can choose $A=1$ and $C=+1,0$ or -1, then 
\[ \Lambda=\left(\renewcommand{\arraystretch}{1.5}\begin{tabular}{c|c}
$D$ &$ \frac{1}{2}\mathbf{C}^{T}$\\ \hline
$\frac{1}{2}\mathbf{C}$ & $\Lambda_{0}$  
\end{tabular}\right),
 \quad \mathbf{C}=(C,C,C)^{T},\ \Lambda_{0}= \begin{pmatrix}
0 & \frac{1}{2} & \frac{1}{2}\\ 
\frac{1}{2} & 0&\frac{1}{2}\\
\frac{1}{2} & \frac{1}{2}& 0  \\  
\end{pmatrix}.
 \]
 
we have 
\[ \det \Lambda = \frac{4D-3C^{2}}{16},\ \text{rk}(\Lambda)=4, \det \Lambda_{0}=\frac{1}{4} .\]
The eigenvalues of $\Lambda_{0}$ are $\{ -\frac{1}{2}, -\frac{1}{2},1\}$. By Bromwich-Burighton's classification,  this quadric belongs to the hyperboloid class ($\det \Lambda_{0}\ne0,\  \text{rk}(\Lambda)=4$) and
\begin{itemize}
\item if $\det \Lambda >0 \iff \frac{4}{3}D>C^{2}$ the hyperboloid is one-sheeted,
\item if $\det \Lambda <0 \iff \frac{4}{3}D<C^{2}$ the hyperboloid is two-sheeted. The position of the hyperbolic sheets with respect to the first octant will be given, as previously, by the number of positive roots of the equation: 
$ 3X^2+3CX+D=0$.

\item if  $\det \Lambda =0 \iff \frac{4}{3}D=C^{2}$ the surface is singular.
\end{itemize}

Let us now discuss,  in function of $D$, the nature of the quadric for each value of $C =-1,0,1$ and deduce the corresponding quartics
\begin{enumerate}
\item $C=-1$. 
\begin{itemize} 
\item  $D<0$, the quadric is a two-sheeted hyperboloid and equation~\eqref{E: Intdigsurf} has only one positive root. Moreover,  the intersection of the hyperboloid with the plane $X=a$, in the fundamental domaine $\{X\geq 0,0\leq Y\leq X,0\leq Z\leq X\}$ , tends to the line $Z+Y-1=0$ when $a\to \infty$. Then quartic is of the type: circular-hemihexacron, the section of which by a plane orthogonal to an axis tends to be a unit circle centered on the axis.
\item  $D=0$.  The quadric is a two-sheeted hyperboloid and equation~\eqref{E: Intdigsurf} has one positive roots and one equal to zero. Topologically, the quartic  is the disjoint union of a circular-hemihexacron with the singular point $\{0\}$.
\item  $0<D<\frac{3}{4}$.  The quadric is a two-sheeted hyperboloid and equation~\eqref{E: Intdigsurf} has two positive roots. The quartic is the disjoint union of a topological sphere nested in  a circular-hemihexacron. 

\item  $D=\frac{3}{4}$. Equation~\eqref{E: Intdigsurf} has a double root $X=\frac{1}{2}$, but any straight line  in the fundamental domain intersect the quadric in two points. The quadric is one-sheeted with a singular point $X=Y=Z=\frac{1}{2}$, then the quartic, restricted to the first octant has a singular point $x=y=z=\frac{1}{\sqrt {2}}$. Finally, because of the symmetry the quartic $Q_{1,0,-1} (\frac{3}{4})$ has 8 singular points, at the vertices of a cube centered at the origin. Moreover, it is two-sheeted except at the singular points. 
inscribed in the sphere of radius $\sqrt\frac{3}{2}$ with faces parallel to the coordinate planes. 
\item  $\frac{3}{4}<D<1$. The quadric is a one-sheeted hyperboloid with axis $X=Y=Z$ and center at a point $X_{0}=Y_{0}=Z_{0}=\frac{1}{2}$. By inverse transformation and symmetries the quartic surface is the connected sum  on each vertex of a cube and a circular-hemihexacron, then the surface presents 8 holes.
\item $D=1$. The quadric is a one-sheeted hyperboloid with axis the oriented line $X=Y=Z$ and center at $X_{0}=Y_{0}=Z_{0}=\frac{1}{2}$. The quartic has 12 singular points.
\item $D>1$. The quadric is a one-sheeted hyperboloid with axis $X=Y=Z$ and center at $X_{0}=Y_{0}=Z_{0}<0$. Its intersection with the first octant is made of 3 sheets. Then coming back to the quartic we obtain 6 smooth disjoint surfaces, more precisely 3 two-sheeted surfaces, having  an coordinate axis as symmetry axis.

\end{itemize}

\vspace{3pt}
\item $C=0$
\begin{itemize}  
\item $D<0$ , $\det \Lambda<0$ the hyperboloid is two sheeted, but only one sheet intersects the first octant. Finally the quartic  is an unbounded six branches star, branches being asymptotic to the coordinates axis.
\item $D=0$ The quadric is singular and reduced to the coordinates axis.
\end{itemize}

\vspace{3pt}
\item $C=1$.
\begin{itemize}  
\item $D<0$ , $\det \Lambda<0$ the hyperboloid is two sheeted, but only one sheet intersects the first octant. Finally the quartic  is an unbounded six branches star, branches being asymptotic to the coordinates axis.
\item $D=0$ The quadric is singular at zero, by calculation. 
The surface is the disjoint union of a connected component centered at the origin with six cylindrical branches tending to infinity.
\item $0<D<\frac{4}{3}$ implies that the hyperboloid is two sheeted. Moreover, both sheets exist in the fundamental domain. Therefore, the surface has two connected concentric components centered at zero. 
\end{itemize}

\end{enumerate}\qed


\subsection{Proof of theorems~\ref{T:1B0D} }
In the case $A\ne 0$, $B\ne0$ and $C=0$ we can choose $A=1$ and $B=-b$ or $b$, $b>0$, then 
\[ \Lambda=\left(\renewcommand{\arraystretch}{1.5}\begin{tabular}{c|c}
$D$ &$ \mathbf{0}^{T}$\\ \hline
$\mathbf{0}$ & $\Lambda_{0}$  
\end{tabular}\right),
\ \Lambda_{0}= \begin{pmatrix}
B & W& W\\ 
W& B&W\\
W & W& B  \\  
\end{pmatrix},\ W=\frac{1+2B}{2}.
 \]
 
\[
\det \Lambda = D\frac{1+3B}{4},\ \det \Lambda _{0}= \frac{1+3B}{4}=\begin{cases}> 0 &\text{ if } B> -\frac{1}{3}\\ 0 &\text{ if } B= -\frac{1}{3}\\< 0 &\text{ if } B< -\frac{1}{3}\end{cases},\ \text{rk}(\Lambda)=4
\]

The position of the quadric sheets with respect to the first octant will be given, as previously, by the number of positive roots of the equation: 
$ 3(1+B)X^2+D=0$.
\begin{itemize}

\item If $B<-\frac{1}{3}$, 
\begin{itemize}
\item $D> 0$ then $\det \Lambda \ne0$ and $\text{rk}(\Lambda)=4$,  the quadric is a real ellipsoid. The quadric is a cuboid (topological sphere).
\item $D=0$, then $\det \Lambda \ne0$ but $\det \Lambda_{0}\ne0$ and $\text{rk}(\Lambda)=3$, the quadric and the the quartic degenerate in $\{0\}$.
\end{itemize}

\item If $B=-\frac{1}{3}$,then $\det \Lambda_{0}=0,\  \det \Lambda=0$ 
\begin{itemize}
\item if $D>0$, $\text{rk}(\Lambda)=3$, we have a improper singular point at infinity so the quadric is a cylinder. The axe of this  cylinder is the line $X=Y=Z$. Coming back to the quartic  we obtain a stellated cube. 
\item if $D=0$, $\text{rk}(\Lambda)=2$, the quadric degenerate in diagonal of cube.
\item if $D<0$, the quartic does not exist in real affine space.
\end{itemize}

\item If $-\frac{1}{3}<B<0$,
 \begin{itemize}
 \item $D< 0$ then $\det \Lambda \ne 0$ and $\text{rk}(\Lambda)=4$, the quadric is two-sheeted with center at zero and with vertex situated on the straight line $\{x=y=z\}$.
Applying the previous arguments of symmetry and change of coordinates the quadric is hence composed of eight hyperbolic sheets atthe vertices of a cube. 

\item $ D> 0$ then $\det \Lambda \ne 0$ and $\text{rk}(\Lambda)=4$, the quadric is one-sheeted hyperboloid, centered in the origin. The intersection with the fundamental domain provides one smooth sheet which does not contain the origin. Finally, the quartic is disjoint sum of six smooth cones with axis being the coordinates axis.

\item $D=0$, then $\det \Lambda \ne0$ but $\det \Lambda_{0}\ne0$ and $\text{rk}(\Lambda)=3$, the quadric and the quartic give the point $\{0\}$.
\end{itemize}

\item If $B>0$, the angle between the asymptotes of the hyperbolic sheets with vertices placed on $\{x=y=z\}$ is strictly superior to $\frac{\pi}{2}$. By this argument the sheet intersects transversally the fundamental domain, separating it into two disjoint sets.
  Using the same arguments as before the quartic is a compact topological sphere. 
\end{itemize}


\subsection{Proof of theorem~\ref{T:1bCD} }
We have  to study the surface defined by equation~\eqref{E:++} and equation~\eqref{E:+-} the corresponding quadric being
\[ F_{++}=\mathbf{ X}^{T}\Lambda^{++}\ \mathbf{ X}=0,\quad
F_{+-}=\mathbf{ X}^{T}\Lambda^{+-}\ \mathbf{ X}=0,
\]
where
\[\Lambda^{++}=\left(\renewcommand{\arraystretch}{1.5}\begin{tabular}{c|c}
$\frac{k}{4}$ &$ \frac{1}{2}\mathbf{I}^{T}$\\ \hline
$\frac{1}{2}\mathbf{I}$ & $\Lambda^{+}_{0}$  
\end{tabular}\right), \Lambda^{+-}=\left(\renewcommand{\arraystretch}{1.5}\begin{tabular}{c|c}
$\frac{k}{4}$ &$ -\frac{1}{2}\mathbf{I}^{T}$\\ \hline
$-\frac{1}{2}\mathbf{I}$ & $\Lambda^{+}_{0}$  
\end{tabular}\right),
 \quad \mathbf{I}=(1,1,1)^{T},\]
 \[ \Lambda^{+}_{0}= \begin{pmatrix}
1 & W  & W\\ 
W & 1 & W\\
W & W & 1  \\  
\end{pmatrix},\quad W=\frac{\beta+2}{2},\ \beta>0,\quad \det \Lambda^{-}_{0}=\frac{\beta^{2}}{4}(\beta+3).
 \]
\[\det \Lambda_{++}=\det \Lambda_{+-}=\frac{\beta^2}{16}(k(\beta+3)-3) \]
The eigenvalues of $\Lambda_{0}^{+} $ are $\{\lambda_{1}^{+}=\lambda_{2}^{+}=-\frac{\beta}{2}, \lambda_{3}^{+}=\beta+3\}$, so $\sigma^{-}_{0}=2,\ \sigma^{+}_{0}=1$. Then the quadric is of hyperboloid type, more precisely
\begin{itemize}
\item One-sheeted hyperboloid if $k>\frac{3}{\beta+3}$ ($\det  \Lambda^{++}=\det  \Lambda^{+-}>0$)  
\item Two-sheeted hyperboloid if  $k<\frac{3}{\beta+3}$ ($ \det  \Lambda^{++}=\det  \Lambda^{+-}<0$)
\end{itemize}

 Moreover the center of the quadric
\begin{itemize} 
\item $F_{++}$ in the coordinate axis is $(-\frac{\sqrt{3}}{2(\beta+3)},-\frac{\sqrt{3}}{2(\beta+3)},-\frac{\sqrt{3}}{2(\beta+3)})$, \item $F_{+-}$ in the coordinate axis is $(\, \frac{\sqrt{3}}{2(\beta+3)},\, \frac{\sqrt{3}}{2(\beta+3)},\,  \frac{\sqrt{3}}{2(\beta+3)})$.
\end{itemize}

As we have illustrated the method in  previous theorems we limit ourself to some proposition enlightening the typical structure of quartic surfaces

\subsubsection{Some properties of quartic surface $Q^{++}(k)$}

\begin{lemma}\label{L:f++}
The octahedral quartic  $Q^{++}(k)=\{ f_{++}=0\}$, is a connected compact surface homeomorphic to a sphere if $k<0<\frac{3}{3+\beta}$.
\end{lemma}
\begin{proof} If  $k<\frac{3}{3+\beta}$,  $F_{++}(k)=0$ is a two-sheeted hyperboloid. 
The intersection of the quadric $F_{++}(k)=0$ with the line $\{X=Y=Z\}$ is the set of roots of the following polynomial: $3(\beta+3) X^{2}+3X+\frac{k}{4}=0$. There is only one positive root. Therefore, the quadric has only one sheet in the fundamental domain. 
We claim that the resulting quartic surface is a compact connected component.  
Indeed, consider in the fundamental domain the intersection of a vertical plane to the hyperboloid and the hyperboloid : 
one obtains a hyperbola. Then, it is easy to compute the angle between the asymptotes and notice that it is greater than $\frac{\pi}{2}$.

If $0<k<\frac{3}{3+\beta}$, the quadric is not defined in the fundamental domain.  Hence, the quartic surface does not exist.
If $k>\frac{3}{3+\beta}$,  $F_{++}(k)=0$ is a one-sheeted hyperboloid.
However, the quartic is the empty set, since the quadric does not intersect the fundamental domain.
\end{proof}
\subsubsection{Some properties of the quadric surface $Q^{+-}(k)$ }

\begin{lemma}\label{L:f+-1}
The octahedral quartic  $Q^{+-}(k)=\{ f_{+-}=0\}$, is a connected compact surface homeomorphic to a sphere if $k<0$.
\end{lemma}
\begin{proof} Proceeding as for lemma~\ref{L:f++}, but with equation $3(\beta+3)X^2-3X+\frac{k}{4}=0$ which has one positive root  for$ k<0$, the quartic is a topological sphere.
\end{proof}
\begin{lemma}\label{L:f+-2}
The octahedral quartic  $Q^{+-}(k)=\{ f_{+-}=0\}$ with $0<k<\frac{3}{3+\beta}$ is the disjoint union of two embedded  compact connected component with octahedral symmetry, each homeomorphic to  sphere .
\end{lemma}
\begin{proof}
The  inequality $k<\frac{3}{3+\beta}$ implies $\det \Lambda^{+-}<0$ and the quadric $F_{+-}=0$ is a two-sheeted hyperboloid. As the polynomial equation $3(\beta+3)X^2-3X+\frac{k}{4}=0$ has two positive roots  if $0<k<\frac{3}{3+\beta}$, the two sheets of the hyperboloid intersect the fundamental domain. Moreover $(X-\frac{1}{2})^{2} -\frac{1-k}{4}=0$ has two positive roots for $0<k<\frac{3}{3+\beta}$. Using the inverse change of variables and the symmetry actions, the reconstructed quartic has a topological sphere centered at the origin disjoint from the other connected compact components with octahedral symmetry.
\end{proof}
\begin{lemma}\label{L:f+-3}
The octahedral quartic  $Q^{+-}(k)=\{ f_{+-}=0\}$ is
\begin{enumerate}
\item a connected compact surface if $\frac{3}{3+\beta}<k<\frac{4}{4+\beta}$. Moreover, it is a double surface with eight holes centered on the diagonals of a cube.
\item a singular compact connected surface with twelve singularities at the center of edges of an octahedron if  $k=\frac{4}{4+\beta}$.
\item the disjoint union of six  compact connected  centered at the vertices of an octahedron if $k>\frac{4}{4+\beta}$.
\end{enumerate} 
\end{lemma}

\begin{proof}
The  inequality $k>\frac{3}{3+\beta}$ implies $\det \Lambda^{+-}>0$ and the quadric $F_{+-}=0$ is a one-sheeted hyperboloid, with center situated in the fundamental domain. 
Notice that the vertical section of this hyperboloid is a hyperbola, with asymptotes that form an angle smaller than $\frac{\pi}{2}$. This implies that the hyperboloid intersects the reflection planes in two disjoint curves. 
 In conclusion, the quartic appears to be a double surface with eight holes centered on the diagonals of a cube.
As $k=\frac{4}{4+\beta}$, the horizontal section of the hyperboloid, an ellipse, is tangent to the plane $\{z=0\}$ in one point: $(\frac{1}{\sqrt{4+\beta}},0,\frac{1}{\sqrt{4+\beta}})$. 
 The reconstruction of the quartic implies that there are twelve singularities. 
 Hence, the quartic being singular for $k=\frac{4}{4+\beta}$ has twelve singular conic points at the center of edges of an octahedron.
 
\
For $k>\frac{4}{4+\beta}$ there is no intersection of the quadric with the plane $Y=X$, but the quadric intersects the fundamental domain in the neighborhood of the coordinate axis. In particular, the intersection is a compact set. Thus, the quartic is the union of six compact connected surfaces homeomorphic to  spheres placed at the vertices of an octahedron.

\end{proof}


\subsection{Proof of theorem~\ref{T:1-bCD} }
We have  to study the surface defined by equation~\eqref{E:--} and equation~\eqref{E:-+} the corresponding quartic being
\[
F_{--}=\mathbf{ X}^{T}\Lambda^{--}\ \mathbf{ X}=0,\quad F_{-+}=\mathbf{ X}^{T}\Lambda^{-+}\ \mathbf{ X}=0,
\]
where
\[\Lambda^{--}=\left(\renewcommand{\arraystretch}{1.5}\begin{tabular}{c|c}
$-\frac{k}{4}$ &$- \frac{1}{2}\mathbf{I}^{T}$\\ \hline
$-\frac{1}{2}\mathbf{I}$ & $\Lambda^{-}_{0}$  
\end{tabular}\right), \Lambda^{-+}=\left(\renewcommand{\arraystretch}{1.5}\begin{tabular}{c|c}
$-\frac{k}{4}$ &$ \frac{1}{2}\mathbf{I}^{T}$\\ \hline
$\frac{1}{2}\mathbf{I}$ & $\Lambda^{-}_{0}$  
\end{tabular}\right),
 \quad \mathbf{I}=(1,1,1)^{T},\]
 \[ \Lambda^{-}_{0}= \begin{pmatrix}
-1 & W  & W\\ 
W & -1 & W\\
W & W & -1  \\  
\end{pmatrix},\quad W=\frac{\beta-2}{2},\ \beta>0,\quad \det \Lambda^{-}_{0}=\frac{\beta^{2}}{4}(\beta-3).
 \]
\[\det \Lambda^{-+}=\det \Lambda^{--}=-\frac{\beta^2}{16}(k(\beta-3)+3);\quad\det \Lambda^{-+}=\det \Lambda^{--}\begin{cases} <0,& k<\frac{3}{3-\beta}\\ =0,& k=\frac{3}{\beta-3}\\>0,& k>\frac{3}{3-\beta}\end{cases}, \beta\ne0
\]

The eigenvalues of $\Lambda_{0}^{-} $ are $\{\lambda_{1}^{-}=\lambda_{2}^{-}=-\frac{\beta}{2},\  \lambda_{3}^{-}=\beta-3\}$, then

\hspace{1cm} $\sigma^{-}_{0}=3,\ \sigma^{+}_{0}=0$ if $\beta<3$,

\hspace{1cm}  $\sigma^{-}_{0}=2,\ \sigma^{+}_{0}=1$ if $\beta>3$.

It follows that for the quadric $F_{--}$ ad $F_{-+}$ we have the following properties 
\begin{itemize}
\item if $k>\frac{3}{3-\beta}$ and $0<\beta<3$ there is no real solution,
\item if $k<\frac{3}{3-\beta}$ and $0<\beta<3$  then the quadric is an ellipsoid,
\item if $k<\frac{3}{3-\beta}$ and $\beta>3$ then the quadric is a one-sheeted hyperboloid, 
\item if $k>\frac{3}{3-\beta}$ and $\beta>3$ then the quadric is a two-sheeted hyperboloid.
\end{itemize}

The center of the  quadric
\begin{itemize} 
\item $F_{--}$ in the coordinate axis is $(\frac{\sqrt{3}}{2(\beta-3)},\frac{\sqrt{3}}{2(\beta-3)},\frac{\sqrt{3}}{2(\beta-3)})$, \item $F_{-+}$ in the coordinate axis is $(-\frac{\sqrt{3}}{2(\beta-3)},\ -\frac{\sqrt{3}}{2(\beta-3)},\ -\frac{\sqrt{3}}{2(\beta-3)})$.
\end{itemize}

We will not give a detailed proof of theorem 5 since the method of proof is the same as previous theorems. However we give some lemmas which show the connectivity of the surface.
Moreover, we do not analyze in details the intersection set of the quadrics with the fundamental domains. The explicit calculations are omitted since they are not very constructive.

\subsubsection{Some properties of quartic surface $Q^{--}(k)$}

\begin{lemma}\label{L:f--}

Let the octahedral quartic be of type $f_{--}$. 

\begin{itemize}

\item If $k<\frac{3}{3-\beta}$ and $0<\beta<3$ then the surface is compact.
\item if $\beta=3$ and $k<0$ the surface is an eight branched star.
\item $3<\beta<4$ 
\begin{list}{$\triangleright$}{}
\item $k<\frac{3}{3-\beta}$,   $Q^{--}(k)=(\#_{i=1}^{8}\mathcal{H}_{i})\#\mathcal{C}$, $\#$ The surface is the connected sum of a cube with hyperbolic sheets having the diagonal of the cube as axis.
\item $\frac{3}{3-\beta}=k$, $Q^{--}(\frac{3}{3-\beta})$, singular quartic, analogous to Cayley's cubic. The quartic is a cube with cones at each vertex.  This surface has eight conical singular points at the vertices of a cube.

\item $\frac{3}{3-\beta}<k<0$, $Q^{--}(k)=(\sqcup_{i=1}^{8}\mathcal{H}_{i})\sqcup\mathcal{C}$. 
\item $k=0$,  $Q^{--}(k)=(\sqcup_{i=1}^{8}\mathcal{H}_{i})\sqcup\{0\}$.
\item $\frac{3}{3-\beta}<0<k$, $Q^{--}(k)=\sqcup_{i=1}^{8}\mathcal{H}_{i}$ 
\end{list}
\end{itemize}
\end{lemma}

\begin{proof}
The proof is based on the results given previously.
If $k<\frac{3}{3-\beta},  \det \Lambda_{0}\neq 0$ and $0<\beta<3$, then the quadric is an ellipsoid, which intersects the fundamental domain and has symmetry axis $\{X=Y=Z\}$. Therefore the quartic is a cuboid.   

If $\beta=3$ and $k<0$, the quadric surface is degenerated, classified as an elliptic cylinder.
It's vertex has coordinate $(\frac{-k}{12},\frac{-k}{12},\frac{-k}{12})$. Thus, as $k<0$ the surface is a star with eight branches.
If $3<\beta<4, k<\frac{3}{3-\beta}$, the quadric is a one-sheeted hyperboloid. The result comes straight-forward.
If $3<\beta<4, \frac{3}{3-\beta}=k$ the surface has eight singularities. In the fundamental domain, this corresponds to a degenerated quadric, which is a cone with vertex lying on the axis $\{X=Y=Z\}$.
If $\frac{3}{3-\beta}<k<0$, the quadric is a two sheeted-hyperboloid with both sheets lying in the fundamental domain. Moreover, the asymptotic lines to the vertical section of the hyperboloid  form an angle smaller than $\frac{\pi}{2}$. Hence, the resulting surface is a compact component surrounded by eight convex sheets. 
If $0<\frac{3}{3-\beta}<k$, the surface is a two sheeted-hyperboloid with only one sheet in the fundamental domain. Therefore, the resulting surface is constituted of eight convex sheets.

\end{proof}


\subsubsection{Some properties of quartic surface $Q^{-+}(k)$}

\begin{lemma}\label{L:f-+}
Let the octahedral quartic be of type $f_{-+}$. 
\begin{itemize}
\item If  $0<k<1$ and $0<\beta<3$ then there exists a quartic surface $ f_{-+}$which is a disjoint union of two concentric compact surfaces, each homeomorphic to a sphere.

\item If $ \beta<3$ and $\frac{4}{4-\beta}<k<\frac{3}{3-\beta}$ then there exists a quartic surface $ f_{-+}$ which is a disjoint union of eight compact surfaces disposed symmetrically about zero at the vertices of a cube.

\item If $k>\frac{3}{3-\beta}$ and $\beta>3$, the quartic surface $ f_{-+}$ has at most eight disjoint convex sheets, each lying symmetrically about zero in their own octant, and with their vertex disposed at the vertices of a cube.

\item If $\beta=3$ then the quartic surface is a union of eight cylinders disposed on the vertices of a cube.
\end{itemize}
\end{lemma}

\begin{proof}
The proof is based on the results given previously and uses the same method as before.
First let us consider how many connected components lie in the positive octant.
The intersection points of $F_{-+}$ with the line are the roots of the degree two polynomial:
$X^2(3\beta-9)+3X-\frac{k}{4}=0.$

In order to have both positive roots, one must have:

$0<k<\frac{3}{3-\beta}$, for $0<\beta<3$. In this case the quadric is an ellipsoid  with positive center: $(-\frac{\sqrt{3}}{2(\beta-3)},\ -\frac{\sqrt{3}}{2(\beta-3)},\ -\frac{\sqrt{3}}{2(\beta-3)})$. 
One can find that the ellipsoid intersects each of the reflection planes $\{X=0\},\{X=0\}, \{Z=0\}$. Therefore, in the fundamental domain there exist two disjoint sheets transversal to the sides of the fundamental domain. By symmetry, the quartic is a disjoint union of compact concentric components, each homeomorphic to a sphere.
\hspace{3pt}

If $k>\frac{3}{3-\beta}, \beta>3$ and $\frac{4}{4-\beta}<k<\frac{3}{3-\beta}$ the family of ellipsoids is strictly contained in the interior of the positive octant. Since the ellipsoid lies on the line $\{X=Y=Z\}$,
the symmetry actions, map this connected component to the seven other octants giving a quartic, constituted of eight connected components.
\hspace{3pt}

The third statement requires that $k>\frac{3}{3-\beta}$ and $\beta>3$. 
This corresponds to a two-sheeted hyperboloid.
Moreover, note that the vertical section of this quadric is a hyperbola with asymptotes that have an angle greater than $\frac{\pi}{2}$. So, the intersection of the hyperboloid with the reflection planes $\{X=0\}, \{Y=0\}, \{Z=0\}$ is empty. Moreover the intersection with the line  $\{X=Y=Z\}$ gives a positive and a negative root. Therefore, the resulting surface is composed of eight convex sheets.
\hspace{3pt}

If $\beta=3$, $det(\Lambda) \neq 0$ and $rk(\Lambda)$=3.
By Bromwitch-Buringhton's classification theorem, the quadric is an elliptic cylinder.
The orientation of the cylinder follows from the eigenvectors and eigenvalues.
By symmetry, the cylinder is mapped to the seven octants.
Thus the quartic is a union of cylinders.
\end{proof}

\begin{corollary}
A degree four octahedral surface has a connected component centered at the origin and homeomorphic to a sphere if we have either cases (1) or (2) and if  (3) is verified:
\begin{enumerate}
\item The quadric equation corresponds to a two sheeted hyperboloid with both sheets in the fundamental polyhedron.
\item The quadric equation in the fundamental polyhedron corresponds to an ellipsoid, with the condition that it's intersection with the planes $\{X=0\},\{Y=0\}$ and $\{Z=0\}$ is an ellipse and that the center is not zero.
\item The intersection of the line \{X=Y=Z\} with the quadric has two positive roots.
\end{enumerate}
\end{corollary}

\begin{proof}the proof follows from the above discussion.\end{proof}

\appendix
\section{Bromwich-Buringthon classification}\label{A:BrBu}

The Bromwich-Buringthon method~\cite{Bro1905,Bur1932,MeYa} of quadrics' classification consists to write the quadratic polynomial 
\[P(x,y,z) =ax^{2}+by^{2}+cz^{2} +2fxy+2gyz+2hzx+2px+2qy+2rz+d.\]
 as a matrix product:
\begin{equation}
P(x,y,z)= \mathbf{X}^{T}\Lambda \mathbf{X}, \quad \Lambda=\begin{pmatrix}d& p&q&r\\p&a&f&h\\q&f&b&g\\r&h&g&c\\\end{pmatrix}, \quad \mathbf{X}=(1,x,y,z)^{T}.
\end{equation}
The method depends on whether we consider real or complex field, and whether we consider the affine or projective space. In this paper we are interested only in the real affine case and moreover by quadric polynomial of the type:
\begin{equation}
F(X,Y,Z)= A(XY+YZ+ZX) +B(X+Y+Z)^{2}+C(X+Y+Z)+D.
\end{equation}
The quadric surface $F(X,Y,Z)=0$ can be  written  as the following matricial equation : 
 
 \begin{equation}\label{E:quartiqueinv}
 \mathbf{ X}^{T}\Lambda\ \mathbf{ X}=0,\quad  \Lambda=\left(\renewcommand{\arraystretch}{1.5}\begin{tabular}{c|c}
$D$ &$ \frac{1}{2}\mathbf{C}^{T}$\\ \hline
$\frac{1}{2}\mathbf{C}$ & $\Lambda_{0}$  
\end{tabular}\right),
\quad \mathbf{ X}=(1,X,Y,Z)^{T},
 \end{equation}
 with
 \[ \Lambda_{0}= \begin{pmatrix}
B & W  & W\\ 
W & B & W\\
W & W & B  \\  
\end{pmatrix},\quad W=\frac{A+2B}{2}, \quad\mathbf{C}=(C,C,C)^{T} .
\]

The geometric classification of the quadric  is given in terms of  matricial  invariants such that: \begin{itemize}
 \item $\det \Lambda$  ($\det \Lambda_{0}$), determinant of $\Lambda$ (resp. $\Lambda_{0}$),
 \item $\text{tr} \Lambda$  ($\text{tr} \Lambda_{0}$), trace of $\Lambda$ (resp. $\Lambda_{0}$),
\item $\sigma^{-}, \sigma^{+}$ ($\sigma^{-}_{0}, \sigma_{0}^{+})$), number of negative and of positive  eigenvalues of $\Lambda$ (resp.  $\Lambda_{0}$),  
\item $\text{rk}(\Lambda)$ ($\text{rk}(\Lambda_{0})$), rank of $\Lambda$ (resp. $\Lambda_{0}$),
\item $J $ sum of $\Lambda_{0}$'s minors.
\end{itemize}

In the case of the quadric~\eqref{E:quartiqueinv} the classification can be made with the matrix invariants
\[\det \Lambda,\ \det \Lambda_{0},\ (\sigma^{-}, \sigma^{+}), \text{rk}(\Lambda),\ J=9(B^{2}-W^{2}).\]
We briefly recall the classification theorem due to Burighton.
\begin{theorem}[Bromwich-Buringhton's classification of quadrics in terms of matrix invariants]

\ 
\vspace{5pt}
\begin{enumerate}

\item If $\det \Lambda_{0}=0 $, then we have the following cases :
\begin{itemize}
\item $\text{rk}(\Lambda)=4$, 
\begin{itemize}
\item $J>0$ elliptic paraboloid
\item $J< 0$ Hyperbolic paraboloid
\end{itemize}
\hspace{5pt}

\item $\text{rk}(\Lambda)=3$
\begin{itemize}
\item $J>0$ Real elliptic cylinder
\item $J< 0$ Hyperbolic cylinder
\item $J=0$  Parabolic cylinder
\end{itemize}
\hspace{5pt}

 \item$\text{rk}(\Lambda)=2$
\begin{itemize}
\item $J>0$ Pair of imaginary lines
\item $J< 0$ Pair of secant planes
\item $J=0$  Pair of parallel plans imaginary or real.
\end{itemize}
\hspace{5pt}

\item $\text{rk}(\Lambda)=1$
\begin{itemize}
\item double plane

\end{itemize}
\end{itemize}
\hspace{5pt}

\item If $\det \Lambda_{0}\ne0 $, then we have the following cases :
\hspace{5pt}

\begin{itemize}
\item $\text{rk}(\Lambda)$=4
\begin{itemize}
\item $(\sigma^{-}_{0}=3,\sigma^{+}_{0}=0)$
\begin{itemize}
\item $\det(\Lambda)>0$ imaginary ellipsoid
\item $\det(\Lambda)<0$ real ellipsoid
\end{itemize}
\item $(\sigma^{-}_{0}=2,\sigma^{+}_{0}=1)$
\begin{itemize}
\item $\det(\Lambda)>0$ one sheeted hyperboloid
\item $\det(\Lambda)<0$ two sheeted hyperboloid
\end{itemize}
\end{itemize}

\item $\text{rk}(\Lambda)=3$
\begin{itemize}
\item $(\sigma^{-}_{0}=3,\sigma^{+}_{0}=0)$ imaginary elliptic cone
\item $(\sigma^{-}_{0}=2,\sigma^{+}_{0}=1)$ real elliptic cone
\end{itemize}
\end{itemize}

\item $\det \Lambda=0$ 
\begin{itemize}
\item rk$(\Lambda)=3$ , the quadric has an isolated singular point
\begin{itemize}
\item If the point is proper then the quadric is a cone with as vertex this point.
\item If the point is improper the quadric is a cylinder.
\end{itemize}
\item if rk$(\Lambda)=2$, the quadric is the intersection of two planes; the singular locus is the intersection set.
\item if rk$(\Lambda)=1 $ the quadric is a plane of multiplicity two; the singular locus is the plane.
\begin{remark}
If  the singularity of a quadric gives a singularity of the corresponding quartic, the converse is not true.
\end{remark}
\end{itemize}
\end{enumerate}

\end{theorem}


\end{document}